\documentclass{wiki}

\title{On the one-dimensional continuity equation with a nearly incompressible vector field}
\author{\it Nikolay A. Gusev\footnote{Steklov Mathematical Institute of
Russian Academy of Sciences; email: \texttt{\href{mailto:n.a.gusev@gmail.com}{n.a.gusev@gmail.com}}}}
\date{\today}

\begin{document}
\maketitle

\begin{abstract}
We consider the Cauchy problem for the continuity equation with a bounded nearly incompressible vector field $b\colon (0,T) \times \R^d \to \R^d$, $T>0$. This class of vector fields arises in the context of hyperbolic conservation laws (in particular, the Keyfitz-Kranzer system).

It is well known that in the generic multi-dimensional case ($d\ge 1$) near incompressibility is sufficient for existence of bounded weak solutions, but uniqueness may fail (even when the vector field is divergence-free), and hence further assumptions on the regularity of $b$ (e.g. Sobolev regularity) are needed in order to obtain uniqueness.

We prove that in the one-dimensional case ($d=1$) near incompressibility is sufficient for existence \emph{and} uniqueness of \emph{locally integrable} weak solutions. We also study compactness properties of the associated Lagrangian flows.
\end{abstract}

\section{Introduction}

Let $b\in L^\infty(I\times \R^d ; \R^d)$ denote a time-dependent vector field on $\R^d$, where $I=(0,T)$, $T>0$, $d\in \mathbb N$. Consider the Cauchy problem for the continuity equation
\begin{equation}\label{Cauchy-problem}
\left\{
\begin{aligned}
&\partial_t u + \mathop{\mathrm{div}_x} (u b) = 0 \qquad \text{in $I\times \R^d$},
\\
&u|_{t=0} = \bar u \qquad\text{in $\R^d$},
\end{aligned}
\right.
\end{equation}
where $\bar u \in L^1_\loc(\R^d)$ is the initial condition and $u\colon I\times \R^d\to \R$ is the unknown. A function $u\in L^1_\loc(I\times \R^d)$ is called \emph{weak solution of \eqref{Cauchy-problem}} if it satisifies \eqref{Cauchy-problem} in sense of distributions:
\begin{equation*}
\int \int u(t,x) (\d_t \varphi(t,x) + b(t,x) \d_x \varphi(t,x)) \, dx \, dt + \int \bar u(x) \varphi(0, x) \, dx = 0
\end{equation*}
for any $\varphi \in C^1_c([0,T)\times \R^d)$.

Existence and uniqueness of weak solution of \eqref{Cauchy-problem} are well-known when
the vector field $b$ is Lipschitz continuous. However in connection with many problems in mathematical physics one has to study \eqref{Cauchy-problem} when $b$ is non-Lipschitz (in general). In particular, vector fields with Sobolev regularity arise in connection with fluid mechanics \cite{DiPernaLions}, and vector fields with bounded variation arise in connection with nonlinear hyperbolic conservation laws \cite{AmbrosioBV}. Therefore one would like to find the weakest assumptions on $b$ under which weak solution of \eqref{Cauchy-problem} exists and is unique.

For a generic bounded vector field $b$ concentrations may occur and therefore the Cauchy problem \eqref{Cauchy-problem} can have no bounded weak solutions. However under mild additional assumptions on $b$ existence of bounded weak solutions can be proved. Namely, the following class of vector fields has been studied in connection with the so-called Keyfitz-Kranzer system (introduced in \cite{KeyfitzKranzer}):

\begin{definition}\label{def-ni}
A vector field $b\in L^\infty(I\times \R^d ; \R^d)$ is called \emph{nearly incompressible with density $\rho \colon I\times \R^d \to \R$} if there exists $C>0$ such that $1/C \le \rho(t,x) \le C$ for a.e. $(t,x)\in I \times \R^d$ and $\rho$ solves 
$\partial_t \rho + \mathop{\mathrm{div}_x} (\rho b) = 0$
(in sense of distributions).
\end{definition}

It is well-known that near incompressibility is sufficient for existence of bounded weak solutions of \eqref{Cauchy-problem}.
However in the generic multidimensional case ($d \ge 2$) 
it is not sufficient for uniqueness. For example, there exists a bounded divergence-free autonomous vector field on the plane ($d=2$), for which \eqref{Cauchy-problem} has a nontrivial
bounded weak solution with zero initial data \cite{ABC2}.

Uniqueness of weak solutions has been established for some classes of weakly differentiable vector fields \cite{DiPernaLions,AmbrosioBV}.
Recently new uniqueness results were obtained for continuous vector fields \cite{Crippa,Shaposhnikov} (without explicit assumptions on weak differentiablilty). Note that in general a nearly incompressible vector field does not have to be continuous (and vice versa). Uniqueness of locally integrable weak solutions has been proved in \cite{CarCri16}
for Sobolev vector fields under additional assumption of continuity.

Uniqueness of bounded weak solutions for nearly incompressible vector fields in the two-dimensional case ($d=2$) was also studied in \cite{BBG2016}. In particular it was proved that uniqueness holds when $b\ne 0$ a.e., or when $b\in BV$.

Our main result is the following:

\begin{theorem}\label{t-main}
Suppose that $b \in L^\infty(I\times \R; \R)$ is nearly incompressible. Then for any initial condition $\bar u\in L^1_\loc(\R)$ the Cauchy problem \eqref{Cauchy-problem} has a \emph{unique} weak solution $u\in L^1_\loc(I \times \R)$.
\end{theorem}

Existence of bounded weak solutions of \eqref{Cauchy-problem} with bounded $\bar u$ for nearly incompressible vector fields is well-known (see e.g. \cite{CrippaThesis} for the case of vector fields with bounded divergence). Uniqueness of bounded weak solutions in the one-dimensional case has already been proved in \cite{BouchotJames98}.
The novelty of Theorem~\ref{t-main} is that it applies to merely locally integrable weak solutions.


\section{Uniqueness of locally integrable weak solutions}

\begin{definition}\label{def-density}
A non-negative function $\rho \in L^1_\loc(I\times \R^d; \R)$ is called a \emph{density associated with a vector field} $b\in L^1_\loc(I\times \R^d; \R^d)$ if $\rho b \in L^1_\loc(I\times \R^d; \R^d)$ and $\d_t \rho + \div(\rho b) = 0$ in $\ss D'(I\times \R^d)$.
\end{definition}

\begin{remark}
If a vector field $b\in L^\infty(I\times \R^d; \R^d)$ admits a density $\rho$ and there exist strictly positive constants $C_1, C_2$ such that $C_1 \le \rho(t,x) \le C_2$ for a.e. $(t,x) \in I\times \R^d$ then $b$ is nearly incompressible.
\end{remark}

Suppose that a vector field $b\in L^1_\loc(I\times \R^d; \R^d)$ admits a density $\rho$.
Since $\partial_t \rho + \d_x (\rho b) = 0$ in $\ss D'((0,T)\times \R)$ there exists $H\in W^{1,1}_\loc((0,T)\times \R)$ such that
\begin{equation} \label{e-dxH}
\d_x H = \rho  \quad \text{and} \quad
\d_t H = - \rho b. 
\end{equation}
in $\ss D'(I\times \R^d)$. 

\begin{definition}
If a function $H\colon I\times \R \to \R$ satisfies \eqref{e-dxH} then it is called a \emph{Hamiltonian associated with $(\rho, b)$}.
\end{definition}

Clearly the Hamiltonian $H$ is unique up to an additive constant. Moreover, if $\rho, b\in L^\infty(I\times \R)$ then the Hamiltonian can be chosen in such a way that it is Lipschitz continuous, i.e. $H\in \Lip([0,T]\times \R)$.

\begin{theorem}\label{th-L1-unique}
Suppose that a vector field $b\in L^\infty(I\times \R; \R)$ admits a density $\rho \in L^\infty_\loc(I\times \R; \R)$ such that $\rho(t,x)>0$ for a.e. $(t,x) \in I\times \R$.
If $u\in L^1_\loc(I\times \R; \R)$ is a weak solution of
\eqref{Cauchy-problem} with $\bar u \equiv 0$ then $u(t,x)=0$ for a.e. $(t,x)\in I \times \R$.
\end{theorem}

\begin{proof}

\emph{Step 1.} 
Let $H\in \Lip([0,T]\times \R)$ be a Hamiltonian associated with $(\rho, b)$.
We would like to use test functions of the form $\varphi(t,x) := f(H(t,x))$ in the districutional formulation of \eqref{Cauchy-problem}, where $f\in C^\infty_c(\R)$. In general such functions could be not compactly supported, therefore we apply an approximation argument.


For any $(t,x)\in(-T,0)\times \R$ let $H(t,x):=H(-t,x)$. 
Clearly $\d_x H = \tilde \rho$ and $\d_t H = \tilde \rho \tilde b$ in $\ss D'((-T,T)\times \R)$, where
\begin{equation*}
\tilde \rho(t,x) := 
\begin{cases}
\rho(t,x), & t > 0, \\
\rho(-t,x), & t<0;
\end{cases}
\qquad
\tilde b(t,x) := 
\begin{cases}
b(t,x), & t > 0, \\
-b(-t,x), & t<0.
\end{cases}
\end{equation*}

Let $\eps > 0$ and let $\omega_\eps(z):= \eps^{-2}\omega(\eps^{-2}z)$,
where $\omega\in C_c^\infty(\R^2)$ is the standard mollification kernel.
Let $H_\eps := H * \omega_\eps$, where $*$ denotes the convolution. Clearly 
\begin{equation}\label{e-dxHeps}
\d_x H_\eps = \tilde \rho_\eps \quad \text{and}\quad \d_t H_\eps = - (\tilde \rho \tilde b)_\eps.
\end{equation}
Hence for any $t\in (-T+\eps,T-\eps)$ the function $H_\eps(t,\cdot)$ is strictly increasing.

\emph{Step 2.}
Let $h\in \R$ be such that the level set $L_{\eps, h} := \{(t,x)\in (-T+\eps,T-\eps)\times \R : H_\eps(t,x)=h\}$ is not empty.

Suppose that $\tau,\xi \in \R$ and $\tau^2 + \xi^2 = 1$. If $|\xi| > \|b\|_\infty |\tau|$ then the derivative of $H$ in the direction $\nu:=(\tau,\xi)$ satisfies
\begin{equation*}
\d_\nu H = \tau \d_t H + \xi \d_x H = -\tau (\tilde \rho \tilde b)_\eps + \xi \tilde \rho_\eps 
\ge (\xi - \tau \|b\|_\infty) \rho_\eps > 0,
\end{equation*}
therefore for any $(t,x)\in L_{\eps, h}$ the level set $L_{\eps, h}$ is contained in some cone:
\begin{equation}\label{e-level-set-in-cone}
L_{\eps, h} \subset \{(t',x') : (x'-x) \le \|b_\infty\| (t' - t)\}.
\end{equation}
Consequently $L_{\eps, h}$ is a bounded subset of $(-T+\eps,T-\eps)\times \R$.

Fix $(t,x)\in L_{\eps, h}$. Since $\d_x H_\eps = \tilde \rho_\eps > 0$, by Implicit Function Theorem
the level set $L_{\eps, h}$ in some neighborhood $U=(t-\delta, t+\delta)\times (x-\delta, x+\delta)$ of $(t,x)$ 
can be represented as a graph of a smooth function $\tau\mapsto Y_\eps(\tau,h)$:
\begin{equation}\label{e-implicit1}
L_{\eps, h} \cap U = \{(\tau, Y_\eps(\tau,h)) \;|\; \tau \in (t-\delta, t+\delta)\}.
\end{equation}
Moreover,
\begin{equation}\label{e-implicit2}
\d_\tau Y_\eps(\tau, h) = - \left.\frac{(\d_t H_\eps)(\tau, x)}{(\d_x H_\eps)(\tau, x)}\right|_{x=Y_\eps(\tau, h)}
\quad \stackrel{\eqref{e-dxHeps}}{\Rightarrow} \quad
|\d_\tau Y_\eps| \le \frac{|(\tilde \rho \tilde b)_\eps|}{|\tilde \rho_\eps|} \le \|b\|_\infty.
\end{equation}

Let $P_{\eps,h} := \pi_t (L_{\eps, h})$, where $\pi_t (\tau, x) := \tau$ is the projection on the $t$-axis.
By \eqref{e-implicit1} $P_{\eps,h}$ is open. On the other hand $P_{\eps,h}$ is closed in $(-T+\eps, T-\eps)$, since $L_{\eps, h} = H_\eps^{-1}(h)$ is closed in $(-T+\eps,T-\eps)\times \R$. Therefore $P_{\eps,h} = (-T+\eps, T-\eps)$.



The image $R_\eps:=H_\eps((-T+\eps, T-\eps)\times \R) \subset \R$ is connected, since $H_\eps$ is continuous and $(-T+\eps, T-\eps)\times \R$ is connected. Moreover, since for any $t\in (-T+\eps, T-\eps)$ the function $x \mapsto H(t,x)$ is strictly increasing and continuous, the images $H(t, \R)$ are open and hence
\begin{equation*}
R_\eps = \cup_{t\in (-T+\eps, T-\eps)} H_\eps(t, \R)
\end{equation*}
is open. Therefore $R_\eps$ is an open interval.

We have thus proved that for any $h\in R_\eps$ the level set $L_{\eps, h}$ can be \emph{globally}
represented as a graph of a smooth function $\tau \mapsto Y_\eps(\tau, h)$, where $\tau \in (-T+\eps, T-\eps)$ and moreover $|\d_\tau Y_\eps| \le \|b\|_\infty$ by \eqref{e-implicit2}.






\emph{Step 3.}
Using Fubini's theorem and the distributional formulation of \eqref{Cauchy-problem} one can show that
there exists a Lebesgue-negligible set $N\subset (0,T)$ such that for any $\tau\in (0,T) \setminus N$
the function $x\mapsto \rho(t,x)$ is strictly positive for a.e. $x$ and for all $\varphi \in \Lip_c([0,\tau]\times R)$ it holds that
\begin{equation}\label{e-distrib-Cauchy}
\int_\R u(\tau,x) \varphi(\tau, x) \, dx - \int_\R \bar u(x) \varphi(0, x) \, dx
= \int_0^\tau \int_\R u \cdot (\d_t \varphi + b \d_x \varphi) \, dx \, dt
\end{equation}

Let us fix $\tau \in (0,T) \setminus N$ and consider $\eps \in (0, T-\tau)$. By \eqref{e-dxH} the function $x \mapsto H(\tau, x)$ is strictly increasing and continuous. Hence the image $I_\tau := H(\tau, \R)$ is a nonempty open interval. 

Consider $f\in C^\infty_c(I_\tau)$ and let $\varphi_\eps(t,x):= f(H_\eps(t,x))$.
We claim that there exists $\eps_1>0$ and a compact $K\subset [0,\tau]\times \R$ such that
\begin{equation*}
\supp \varphi_\eps \subset K
\end{equation*}
for any $\eps \in (0, \eps_1)$.

Indeed, the support of $f$ is contained in some finite interval $(\alpha, \beta)$ such that $[\alpha, \beta] \subset I_\tau$.
Let us fix $\alpha_1 \in I_\tau \setminus [\alpha, +\infty)$ and $\beta_1 \in I_\tau \setminus (-\infty, \beta]$.
By definition of $I_\tau$ there exist $x_1$ and $y_1$ such that $H(t,x_1) = \alpha_1$ and $H(t,y_1) = \beta_1$.
Since $H_\eps(t,x_1) \to H(t,x_1)$ and $H_\eps(t,y_1)\to H(t,y_1)$ as $\eps \to 0$ we can find $\eps_0 > 0$
such that $R_\eps \supset (\alpha, \beta)$ for any $\eps \in (0,\eps_0)$.

Since $x\mapsto H_\eps(\tau, x)$ is strictly monotone and continuous, there exist unique $x_0$ and $y_0$
such that $H(x_0,\tau) = \alpha$ and $H(y_0, \tau) = \beta$
Since the support of $f$ is a compact subset of $(\alpha,\beta)$ and $H_\eps(x_0,\tau) \to \alpha$ and $H_\eps(y_0, \tau) \to \beta$ as $\eps \to 0$, there exists $\eps_0 >0$ such that 
\begin{equation*}
\supp f \subset (H_\eps(x_0, \tau), H_\eps(y_0, \tau))
\end{equation*}
whenever $\eps \in (0,\eps_0)$.
Hence the support of $\varphi_\eps$ (restricted to $[0,\tau]\times \R$) is confined by the level sets of $H_\eps$, passing through $x_0$ and $y_0$:
\begin{equation*}
\supp \varphi_\eps \subset \{(t,x) \;|\; t\in [0,\tau], \; x\in [Y_\eps(t, H_\eps(\tau, x_0)), Y_\eps(t, H_\eps(\tau, y_0))]\}
\stackrel{\eqref{e-implicit2}}{\subset} K,
\end{equation*}
where
\begin{equation*}
K:=\{(t,x) \;|\; t \in [0,\tau], x \in [x_0 - \|b\|_\infty (\tau - t), y_0 + \|b\|_\infty (\tau - t)]\}.
\end{equation*}

\emph{Step 4.}
Now we are in a position to use $\varphi_\eps$ as a test function in \eqref{e-distrib-Cauchy}.
First we observe that
\begin{equation*}
\d_t \varphi_\eps + b \d_x \varphi_\eps = f'(H_\eps(t,x)) (\d_t H_\eps + b \d_x H_\eps) 
= f'(H_\eps(t,x)) (- (\tilde\rho \tilde b)_\eps + b \tilde \rho_\eps) \to 0
\end{equation*}
a.e. on $(0,\tau)\times \R$ as $\eps \to 0$. Since $\bar u \equiv 0$, by \eqref{e-distrib-Cauchy} and Lebesgue's dominated convergence theorem
\begin{equation}\label{e-limit}
\int_\R u(\tau,x) \varphi_\eps(\tau, x) \, dx
= \int \int_K u \cdot (\d_t \varphi_\eps + b \d_x \varphi_\eps) \, dx \, dt \to 0
\end{equation}
as $\eps \to 0$. (Indeed, $|u \cdot (\d_t \varphi_\eps + b \d_x \varphi_\eps)| \le \|f\|_{C^1} \|\rho\|_{L^\infty(K)} (1 + \|b\|_\infty) |u| \in L^1(K)$.)
Since $H_\eps(\tau, \cdot) \to H(\tau, \cdot)$ uniformly on $[x_0, y_0]$, the left-hand side of the equality above converges to $\int_\R u(\tau, x) f(H(t,x)) \, dx$. We have thus proved that
\begin{equation}\label{e-varid}
\int u(\tau,x) f(H(\tau, x)) \, dx = 0
\end{equation}
for all $f \in C^1_c(I_\tau)$. Approximating $f\in C_c(I_\tau)$ with a sequence of functions from $C^1_c(I_\tau)$ it is easy to see that \eqref{e-varid} holds for any $f \in C_c(I_\tau)$.

Fix $\psi \in C_c(\R)$. Since $x \mapsto H(\tau, x)$ is strictly monotone and continuous, it has a continuous inverse, and therefore we can find $f\in C_c(I_\tau)$ such that $\psi(x) = f(H(\tau,x))$ for all $x\in \R$. Therefore by \eqref{e-varid}
\begin{equation}
\int u(\tau, x) \psi(x) \, dx = 0
\end{equation} 
for any $\psi \in C_c(\R)$. Hence $u(\tau, \cdot) \equiv 0$. Since this argument is valid for any $\tau\in (0,T) \setminus N$, we conclude that $u(\tau, \cdot) = 0$ a.e. for a.e. $\tau\in I$.
\end{proof}

From the proof above one can also deduce the following result:

\begin{theorem}
Suppose that a vector field $b\in L^\infty(I\times \R; \R)$ admits a density $\rho \in L^1_\loc(I\times \R; \R)$ such that $\rho(t,x)>0$ for a.e. $(t,x) \in I\times \R$.
If $u\in L^\infty_\loc(I\times \R; \R)$ is a weak solution of
\eqref{Cauchy-problem} with $\bar u \equiv 0$ then $u(t,x)=0$ for a.e. $(t,x)\in I \times \R$.
\end{theorem}

\begin{proof}[The proof repeats the proof of Theorem~\ref{th-L1-unique}.]
Only when passing to the limit in \eqref{e-limit} we have to argue slightly differently.
Namely, since $\tilde \rho \in L^1_\loc([-T,T]\times \R)$ it follows that
\begin{equation*}
\begin{aligned}
&\|u \cdot (\d_t \varphi_\eps + v \d_x \varphi_\eps)\|_{L^1(K)}
\le \|u\|_{L^\infty(K)} \cdot \|-(\tilde \rho \tilde b)_\eps + b \tilde \rho_\eps\|
\\
&\le \|u\|_{L^\infty(K)} \cdot \bigl( \|-(\tilde \rho \tilde b)_\eps + \tilde \rho \tilde b\|_{L^1(K)} + \|-\tilde \rho \tilde b + b \tilde \rho_\eps\|_{L^1(K)} \bigr) \to 0
\end{aligned}
\end{equation*}
as $\eps \to 0$.
\end{proof}







\section{Lagrangian flows and existence of weak solutions}

Suppose that $b\in L^\infty(I\times \R; \R)$ is a nearly incompressible vector field with density $\rho \in L^\infty(I\times\R; \R)$. Let $H\in \Lip([0,T]\times \R)$ be a Hamiltonian associated with $(\rho,b)$. 

By \eqref{e-dxH} and Fubini's theorem for a.e. $t\in I$ for all $x,y \in \R$ such that $x<y$ it holds that
\begin{equation}\label{e-H-difference1}
C_1 (y-x) \le H(t,y) - H(t,x) \le C_2 (y-x),
\end{equation}
where $C_1, C_2$ are the constants from Definition~\ref{def-ni}.
By continuity of $H$ \eqref{e-H-difference1} holds for all $t\in \bar I$.
Hence for any $t\in \bar I$ the function $x\mapsto H(t,x)$ is strictly increasing and \emph{bilipschitz}.
Consequently, for any $h \in \R$ there exists unique $Y(t,h)\in \R$ such that $H(t,Y(t,h)) = h$.

By \eqref{e-H-difference1} for any $t\in [0,T]$ there exists a function $\rho_t \in L^\infty$ such that $C_1 \le \rho_t \le C_2$ a.e. and
\begin{equation*}
\d_x H(t,x) = \rho_t(x)
\end{equation*}
in $\ss D'(\R)$. Note that by continuity of $H$ the function $I\ni t \mapsto \rho_t \in L^\infty(\R)$ is $*$-weak continuous
and therefore $\rho$ solves the Cauchy problem for the continuity equation \eqref{Cauchy-problem} with the initial data $\rho_0$. In view of \eqref{e-dxH} for a.e. $t\in I$ we have $\rho(t,x) = \rho_t(x)$ for a.e. $x$.
Since we can always redefine $\rho$ on a negligible set, for convenience we will
assume that the last equality holds for \emph{all} $t\in [0,T]$.

\begin{lemma}\label{l-level-sets}
The function $Y$ is Lipschitz continuous on $[0,T]\times \R$.
Moreover, there exists a negligible set $M\subset \R$ such that for all $h\in \R \setminus M$
\begin{equation}\label{e-ODE}
\d_t Y(t,h) = b(t,Y(t,h))
\end{equation}
in $\ss D'(I)$.
Finally, for all $t\in [0,T]$
\begin{equation}\label{e-Y-pushforward}
Y(t,\cdot)_\# \Le = \rho(t, \cdot) \Le.
\end{equation}
\end{lemma}

Here $f_\# \mu$ denotes the image of the measure $\mu$ under the map $f$ and $\Le$ denotes the Lebesgue measure (we use the notation from \cite{AFP}).

\begin{proof}

By \eqref{e-H-difference1} for any $h, h' \in \R$ it holds that
\begin{equation*}
C_1 |Y(t,h) - Y(t,h')| \le |H(t,Y(t,h)) - H(t,Y(t,h'))| = |h-h'|
\end{equation*}
hence the function $h\mapsto Y(t,h)$ is Lipschitz continuous with Lipschitz constant $1/C_1$.

Fix $(t,x) \in I\times \R$.
In view of \eqref{e-dxH} and Fubini's theorem for a.e. $(t',x') \in I \times \R$ such that $|x'-x| > \|b\|_\infty |t'-t|$
it holds that
\begin{equation}\label{e-H-difference}
|H(t',x') - H(t,x)| \ge C_1 (|x'-x| - \|b\|_\infty |t'-t|).
\end{equation}
By continuity of $H$, \eqref{e-H-difference} holds for \emph{all} $(t',x')\in I\times \R$.
Hence for any $h\in \R$ and any $(t,x)\in H^{-1}(h)$ the level set $H^{-1}(h)$ is contained in a cone:
\begin{equation}
H^{-1}(h) \subset \{(t',x')\in I \times \R : |x'-x| \le \|b\|_\infty |t'-t|\},
\end{equation}
therefore for any $h\in \R$ the function $t \mapsto Y(t,h)$ is Lischitz continuous with Lipschitz constant $\|b\|_\infty$.

In view of Rademacher's theorem the functions $H$ and $Y$ are differentiable a.e. on $I\times \R$.
Hence by chain rule and taking into account \eqref{e-dxH}
we obtain
\begin{equation*}
\begin{aligned}
0 &= \d_t h = \d_t H(t, Y(t,h)) = \d_t H(t, Y(t,h)) + \d_x H(t,Y(t,h)) \d_t Y(t,h)\\
&= - \rho(t,Y(t,h)) b(t,Y(t,h)) + \rho(t,Y(t,h)) \d_t Y(t,h). 
\end{aligned}
\end{equation*}
and
\begin{equation*}
\begin{aligned}
1 &= \d_h h = \d_h H(t, Y(t,h)) = \d_x H(t,Y(t,h)) \d_h Y(t,h)\\
&= \rho(t,Y(t,h)) \d_h Y(t,h)
\end{aligned}
\end{equation*}
for a.e. $(t,h)\in I \times \R$.
Hence \eqref{e-ODE} holds and moreover for any $\varphi \in C_c(\R)$ 
\begin{equation*}
\begin{aligned}
\int \varphi \, dY(t,\cdot)_\# \Le &= \int \varphi(Y(t,h)) \, dh \\
&=\int \varphi(Y(t,h)) \rho(t,Y(t,h)) \d_h Y(t,h) \, dh 
= \int \varphi (y) \rho(t,y) \, dy
\end{aligned}
\end{equation*}
(by Area formula, see e.g. \cite{AFP}). Thus \eqref{e-Y-pushforward} is proved.
\end{proof}

We define the \emph{flow} $X$ of $b$ as
\begin{equation}\label{e-X-def}
X(t,x) := Y(t, H(0,x)).
\end{equation}
Note that $X$ is independent of the additive constant in the definition of $H$.
In order to show that $X$ is independent of the choice of $\rho$ we recall the definition of regular Lagrangian flow (see \cite{CrippaThesis}) and the corresponding uniqueness result:

\begin{definition}\label{d-rlf}
Let $b\colon [0,T]\times \R^d \to \R^d$ be a bounded measurable vector field.
We say that a map $X\colon [0,T]\times \R^d \to \R^d$ is a regular Lagrangian flow relative to $b$ if
\begin{enumerate}
\item for $\Le^d$-a.e. $x\in \R^d$ the map $t \mapsto X(t,x)$ is an absolutely continuous integral solution of $\dot \gamma(t) = b(t,\gamma(t))$ for $t\in [0,T]$ with $\gamma(0)=x$;
\item there exists a constant $L>0$ independent of $t$ such that $X(t,\cdot)_\# \Le^d \le L \Le^d$.
\end{enumerate}
\end{definition}

\begin{proposition}[see \cite{CrippaThesis}, Theorem 6.4.1]\label{p-regular-lagrangian-flow}
Let $b\colon [0,T]\times \R^d \to \R^d$ be a bounded measurable vector field.
Assume that the only weak solution $u\in L^\infty(I\times \R^d)$ of \eqref{Cauchy-problem} with $\bar u = 0$ is $u =0$.
Then the regular Lagrangian flow relative to $b$, if it exists, is unique. Assume in addition that \eqref{Cauchy-problem}
with $\bar u = 1$ has a positive solution $u\in L^\infty(I\times \R^d)$. Then we have existence of a regular Lagrangian flow relative to $b$.
\end{proposition}

By Lemma \ref{l-level-sets} the flow $X$ defined in \eqref{e-X-def} is a regular Lagrangian flow of $b$.
Indeed, by \ref{e-Y-pushforward}
\begin{equation}\label{e-X-pushforward}
X_\# (\rho(0,\cdot) \Le)  = Y(t,\cdot)_\# H(0,\cdot)_\# (\rho(0,\cdot)\Le) = Y(t,\cdot)_\# \Le = \rho(t,\cdot) \Le.
\end{equation}
Since Theorem~\ref{th-L1-unique} implies uniqueness of bounded weak solutions of \eqref{Cauchy-problem}, Proposition~\ref{p-regular-lagrangian-flow} immediately implies uniqueness of regular Lagrangian flow of $b$. Hence $X$ is independent of the choice of the density $\rho$.

\begin{theorem}\label{th-existence}
Let $b\in L^\infty(I\times \R; \R)$ be nearly incompressible with the density $\rho$.
Let $X$ be the flow of $b$. Then for any $\bar u\in L^1_\loc(\R)$
there exists a function $u\in L^1_\loc(I\times \R)$ such that for a.e. $t\in I$
\begin{equation*}
u(t,\cdot) \Le = X(t,\cdot)_\# (\bar u \Le)
\end{equation*}
and the function $u$ solves \eqref{Cauchy-problem}.
\end{theorem}

\begin{proof}
It is straightforward to check that for any $t\in [0,T]$ the inverse $X^{-1}(t,\cdot)$ of the function $X(t,\cdot)$
is given by $X^{-1}(t,x) = Y(0,H(t,x))$. We define $u(t,x)$ as follows:
\begin{equation*}
u(t,x):= \frac{\bar u(X^{-1}(t,x))}{\rho(0,X^{-1}(t,x))} \rho(t,x).
\end{equation*}
Then
\begin{equation*}
\begin{aligned}
u(t,\cdot) \Le &= \frac{\bar u (X^{-1}(t,\cdot))}{\rho (0, X^{-1}(t,\cdot))}X_\# (\rho(0,\cdot) \Le)
\\ &= X_\# \left( \frac{\bar u(\cdot)}{\rho(0,\cdot)}\rho(0,\cdot) \Le \right) = X(t,\cdot)_\# (\bar u \Le)
\end{aligned}
\end{equation*}
Therefore for any $\varphi\in C^1_c([0,T)\times \R)$ by Definition~\ref{d-rlf} and Theorem~\ref{th-existence}
\begin{equation*}
\begin{aligned}
&\int_I\int_\R (\d_t \varphi + b \d_x \varphi) u(t,x) \, dx \, dt
= \int_I\int_\R (\d_t \varphi + b \d_x \varphi) dX(t,\cdot)_\# (\bar u \Le) \, dt
\\&= \int_I\int_\R [(\d_t \varphi)(t,X(t,x)) + b(t,X(t,x)) (\d_x \varphi)(t,X(t,x))] \bar u(x) \, dx \, dt
\\&= \int_I\int_\R \d_t(\varphi(t,X(t,x))) \bar u(x) \, dx \, dt
\\&= -\int_\R \varphi(t,x) \bar u(x) \, dx \, dt. \qedhere
\end{aligned}
\end{equation*}
\end{proof}

\begin{proof}[Proof of Theorem~\ref{t-main}]
Existence follows from Theorem~\ref{th-existence} and uniqueness follows from Theorem~\ref{th-L1-unique}.
\end{proof}

\begin{remark}
It would be interesting to study existence and uniqueness of weak solutions of \eqref{Cauchy-problem}
for vector fields admitting non-negative density which may vanish on the sets of positive measure.
Such vector fields (in particular in dimension one) are relevant to the Kuramoto-Sakaguchi equation \cite{AmadoriHaPark}.
\end{remark}

\section{Compactness of flows}

In \cite{Bressan03} Bressan has proposed the following conjecture:

\begin{conjecture}[\cite{Bressan03}]\label{c-Bressan}
Consider a sequence of smooth vector fields $b_n\colon I\times \R^d \to \R^d$
which are uniformly bounded, i.e. $|b_n| \le C$ for some $C>0$ for all $n\in \N$.
Let $X_n=X_n(t,x)$ denote the classical flow of $b_n$, i.e.
\begin{equation*}
X_n(0,x) = x, \qquad \d_t X_n(t,x) = b_n(t,X_n(t,x)).
\end{equation*}
Suppose that there exist constants $C_1$, $C_2$
\begin{equation}
\begin{gathered}
C_1 \le |\det (\nabla_x X(t,x))| \le C_2,  \quad (t,x)\in I\times \R^d, \notag\\
\|\nabla_x b_n\|_{L^1} \le C_3. \label{e-BV-bound}
\end{gathered}
\end{equation}
Then the sequence $X_n$ is strongly precompact in $L^1_\loc(I\times \R^d; \R^d)$.
\end{conjecture}

\begin{theorem}\label{th-compactness}
Consider a sequence of one-dimensional vector fields $b_n \in L^\infty(I\times \R; \R)$
which are uniformly bounded, i.e. $|b_n| \le C$ for some $C>0$ for all $n\in \N$.
Let $X_n=X_n(t,x)$ denote the (regular Lagrangian) flow of $b_n$.
Suppose that for each $n\in \N$ the vector field $b_n$ is nearly incompressible with density $\rho_n$
and there exist constants $C_1, C_2$ such that
\begin{equation*}
C_1 \le \rho_n \le C_2
\end{equation*}
a.e. on $I\times \R$ for all $n\in \N$.
Then the sequence $X_n$ is precompact in $C(K)$ for any compact $K\subset I\times \R$.
\end{theorem}

\begin{proof}
By \eqref{e-X-def} and the estimates from the proof of Lemma~\ref{l-level-sets}
one can easily deduce that for any $n\in \N$
\begin{equation*}
|X_n(t,x) - X_n(t',x')| \le \frac{C_2}{C_1}|x-x'| + \|b\|_\infty |t-t'|
\end{equation*}
for all $x,x'\in \R$ and $t,t'\in[0,T]$. Therefore it remains to apply Arzel\`a-Ascoli theorem.
\end{proof}

\begin{remark}
Theorem~\ref{th-compactness} shows that in the one-dimensional case Conjecture holds even without assuming BV bound \eqref{e-BV-bound}.
A quantitative version of Conjecture~\ref{c-Bressan} assuming only the BV bound \eqref{e-BV-bound} (without near incompressibility) has be established in \cite{StefanoBianchini2006}.
\end{remark}

\section{Acknowledgements}
This work is supported by the Russian Science Foundation under grant \No 14-50-00005.
The author would like to thank Debora Amadori, Paolo Bonicatto, Fran\c{c}ois Bouchot and Gianluca Crippa for interesting discussions of this work and their valuable remarks.  

\bibliographystyle{unsrt}
\bibliography{bibliography}

\end{document}